\documentclass{amsart}
\usepackage{amssymb,graphicx}
\usepackage{mathrsfs}
\usepackage{color}
\usepackage{enumerate}
\usepackage{a4wide}
\usepackage[colorlinks=true,citecolor=blue,linkcolor=blue]{hyperref}

\newcommand{\Rl}{\mathbb{R}}
\newcommand{\Cplx}{\mathbb{C}}

\newcommand{\Bc}{\mathcal{B}}

\newcommand{\Sc}{\mathcal{S}}

\newcommand{\sa}{{\mathrm{sa}}}

\numberwithin{equation}{section}

\newtheorem{theorem}{Theorem}[section]

\newtheorem{definition}[theorem]{Definition}
\newtheorem{lemma}[theorem]{Lemma}

\theoremstyle{remark}

\title[$n$-times Fr\'echet differentiability 
on $\Sc^p,$]{Necessary and sufficient conditions 
for $n$-times Fr\'echet differentiability 
on $\Sc^p,$ $1 <p<\infty.$}

\author[C.~Le Merdy]{Christian Le Merdy}
\email{clemerdy@univ-fcomte.fr}
\address{Laboratoire de Math\'ematiques de Besan\c con, UMR 6623, 
CNRS, Universit\'e Bourgogne Franche-Comt\'e,
25030 Besan\c{c}on Cedex, FRANCE}

\author[E.~McDonald]{Edward McDonald}
\email{edward.mcdonald@unsw.edu.au}
\address{School of Mathematics \& Statistics, University of NSW,
Kensington NSW 2052, AUSTRALIA}


\date{\today}

\begin{document}

\maketitle{}

\begin{abstract}
Let $1<p<\infty$ and let $n\geq 1$.
It is proved that a function $f:\Rl\to \Cplx$ is $n$-times Fr\'echet 
differentiable on $\Sc^p$ at every self-adjoint operator
if and only if $f$ is $n$-times differentiable, $f',f'',\ldots,f^{(n)}$ 
are bounded and $f^{(n)}$ is uniformly continuous.
\end{abstract}

\section{Introduction}
Let $H$ be an infinite dimensional separable 
Hilbert space. For $1<p<\infty$, denote by $\Sc^p:=\Sc^p(H)$ the $p$-Schatten class
on $H$, with norm $\Vert\,\cdotp\Vert_p$.
For $n\geq 1$, denote by $C^n_b(\Rl)$ the space of 
$n$-times continuously differentiable functions 
$f:\Rl\to \Cplx$ such that $f',f'',\ldots, f^{(n)}$ are all 
bounded\footnote{Note that $f$ itself is \emph{not} assumed to be bounded.} 
and $C^\infty_b(\Rl) := \bigcap_{n\geq 1} C^n_b(\Rl).$
We denote the self-adjoint (real) 
subspace of $\Sc^p$ by $\Sc^p_{\sa}$.

A famous result of Potapov and Sukochev \cite{PS-acta} states that if 
$f:\Rl\to \Cplx$ is 
Lipschitz continuous, $A$ is a (potentially unbounded) self-adjoint 
operator on $H$ and $X\in \Sc^p_{\sa}$, then
\begin{equation*}
f(A+X)-f(A) \in \Sc^p.
\end{equation*}
It therefore makes sense to define a function
\begin{equation}\label{FunctionPS}
\varphi_{A,p}:\Sc^p_{\sa}\to \Sc^p,\quad \varphi_{A,p}(X) = f(A+X)-f(A).
\end{equation}
The following definitions are in line with \cite[Definition 3.2]{LeMerdySkripka2019}.

\begin{definition}
Let $1<p<\infty$, let $n\geq 1$ and let $f:\Rl\to \Cplx$ be
Lipschitz continuous.

\begin{itemize}
\item [(1)] 
Let $A$ be a potentially unbounded self-adjoint operator on $H$.
  
We say that $f$ is $1$-time Fr\'echet differentiable on $\Sc^p$ at $A$ 
if there exists a bounded linear map
\begin{equation*}
D_p^1f(A) \in \Bc(\Sc^p,\Sc^p),\quad X\mapsto D_pf(A)[X]
\end{equation*}
such that
\begin{equation*}
\|f(A+X)-f(A)-D_p^1f(A)[X]\|_p = o(\|X\|_p)
\end{equation*}
as $\|X\|_p\to 0$, $X\in \Sc^p_{sa}$.
    
For $n\geq 2$, we say that $f$ is $n$-times Fr\'echet differentiable 
on $\Sc^p$ at $A$ if $f$ is $(n-1)$-times Fr\'echet differentiable on $\Sc^p$ at 
$A+X$ for every $X$ in an $\Sc^p_{sa}$-neighborhood of $0$ 
and there exists a bounded $n$-multilinear operator
\begin{equation*}
D^n_pf(A) \in \Bc_n(\Sc^p\times \cdots\times \Sc^p,\Sc^p),\quad (X_1,\ldots,X_n)\mapsto 
D^n_pf(A)[X_1,\ldots,X_n]
\end{equation*}
such that
\begin{align*}
\|D^{n-1}_pf(A+X)[X_1,\ldots,X_{n-1}]-D^{n-1}_pf(A)&[X_1,
\ldots,X_{n-1}]-D^n_pf(A)[X_1,\ldots,X_{n-1},X]\|_p\\
&= o(\|X\|_p)\|X_1\|_p\cdots \|X_{n-1}\|_p
\end{align*}
as $\|X\|_p\to 0,$ $X\in \Sc^p_{sa}$, uniformly for all $X_1,\ldots, X_{n-1}\in \Sc^p$.

\smallskip
\item [(2)] We say that $f$ is $n$-times 
continuously Fr\'echet differentiable on $\Sc^p$ if it is 
$n$-times 
Fr\'echet differentiable on $\Sc^p$ at every self-adjoint operator $A$ and for any 
such $A$, the mapping
$$
\Sc^p_{sa}\to \Bc_{n}(\Sc^p\times \cdots\times \Sc^p,\Sc^p),
\quad X\mapsto D^{n}_pf(A+X),
$$
is continuous.
\end{itemize}
\end{definition}

We note that the function $\varphi_{A,p}$ from (\ref{FunctionPS})
is an $n$-times Fr\'echet differentiable map between the 
Banach spaces $\Sc_\sa^p$ and $\Sc^p$ if and only if 
$f$ is $n$-times 
Fr\'echet differentiable on $\Sc^p$ at $A+X$  for any $X\in \Sc^p_{sa}$. 
In this case,
for any $1\leq m\leq n$,
the $m$th order Fr\'echet derivative of $\varphi_{A,p}$ is the mapping
\begin{equation}\label{Dm}
D^m\varphi_{A,p} = D_p^mf(A+\,\cdotp)
\colon \Sc^p_{sa}\to \Bc_{m}(\Sc^p\times \cdots\times \Sc^p,\Sc^p).
\end{equation}
Consequently $f$ is $n$-times 
continuously Fr\'echet differentiable on $\Sc^p$ if and only
if $\varphi_{A,p}$ is $n$-times continuously Fr\'echet differentiable
for any self-adjoint operator $A$ on $H$.

The aim of this note is to prove the following necessary and 
sufficient characterisation of such functions.

\begin{theorem}\label{main_result}
Let $f$ be a Lipschitz continuous function on $\Rl,$ let $1 < p < \infty$ 
and let $n\geq 1.$ Then the following assertions are equivalent:
\begin{enumerate}[{\rm (i)}]
\item{}\label{continuous_frechet_differentiability} $f$ is $n$-times 
continuously Fr\'echet differentiable on $\Sc^p$,
\item{}\label{frechet_differentiability} For all self-adjoint $A$, 
$f$ is $n$-times Fr\'echet differentiable on $\Sc^p$ at $A$,
\item{}\label{continuity_condition} $f\in C^{n}_b(\Rl)$ and 
$f^{(n)}$ is uniformly continuous on $\Rl.$
\end{enumerate}  
\end{theorem}

The existence of $f\in C^{1}_b(\Rl)$ which 
is not $1$-time Fr\'echet differentiable on $\Sc^p$ at all
self-adjoint $A$ was established in \cite[Example 7.20]{KPPS}.
It is shown in \cite[Theorem 3.6]{LeMerdySkripka2019}
that a Lipschitz continuous function $f:\Rl\to \Cplx$ is $n$-times 
continuously Fr\'echet differentiable at every bounded 
self-adjoint operator if and only if $f\in C^n(\Rl)$ 
(the case $n=1$ goes back to \cite[Theorem 7.17]{KPPS}).
Theorem \ref{main_result} is a similar characterization
in the unbounded case. We refer the reader to the above 
mentioned papers and to \cite{ACDS} for more results on
this theme.

\section{Proof of the main result}
It is clear that \eqref{continuous_frechet_differentiability} 
implies \eqref{frechet_differentiability}. We prove Theorem \ref{main_result} 
by first showing that \eqref{frechet_differentiability} implies \eqref{continuity_condition}, 
and secondly that \eqref{continuity_condition} implies \eqref{continuous_frechet_differentiability}.

\begin{proof}[Proof that \eqref{frechet_differentiability} implies \eqref{continuity_condition}]
Assume that $f$ satisfies (ii). Then it follows from 
\cite[Proposition 3.9]{LeMerdySkripka2019}
that $f\in C^n_b(\Rl)$. Hence, we only
need to prove that $f^{(n)}$ is uniformly continuous. 
Let $\{\lambda_k\}_{k=0}^\infty$ be a dense sequence in $\Rl.$ Let $\{e_k\}_{k=0}^\infty$
be an orthonormal basis for $H$ and define an unbounded 
self-adjoint operator on $H$ by
\begin{equation*}
Ax = \sum_{k=0}^\infty \lambda_k\langle e_k,x\rangle e_k,\quad  
\mathrm{dom}(A) = \Bigl\{x \in H\;:\; \sum_{k=0}^\infty 
|\langle e_k,x\rangle\lambda_k|^2 < \infty\Bigr\}.
\end{equation*}
For $k\geq 0$, let $Q_k$ be the rank one projection on $H$ defined by
\begin{equation*}
Q_kx = e_k\langle e_k,x\rangle,\quad x\in H.
\end{equation*}

We denote $\varphi_{A} : = \varphi_{A,p}$ for brevity.
By assumption, $\varphi_A$ is $n$-times Fr\'echet differentiable 
on $\Sc^p_{sa}$. Repeating identically the argument of 
\cite[Proposition 3.9]{LeMerdySkripka2019}, we obtain that 
\begin{equation}\label{0-derivative_at_Q_k}
\varphi_A(tQ_k) =(f(\lambda_k + t) -f(\lambda_k))Q_k
\quad k\geq 0,\,t \in \Rl,
\end{equation}
and
for any $1\leq m\leq n$, 
\begin{equation}\label{derivative_at_Q_k}
D^m\varphi_A(tQ_k)[Q_k,Q_k,\ldots,Q_k] = 
f^{(m)}(\lambda_k+t)Q_k,\quad k\geq 0,\,t \in \Rl,
\end{equation}
where $D^m\varphi_A$ is 
the $m$th order Fr\'echet derivative of $\varphi_A$, see (\ref{Dm}).
We write $D^0\varphi_A=\varphi_A$ by convention.
Applying \eqref{derivative_at_Q_k} with $m=n$, and either
\eqref{0-derivative_at_Q_k} if $n=1$ or
\eqref{derivative_at_Q_k} with $m=n-1$ if $n\geq 2$, we obtain
\begin{align*}
D^{n-1}\varphi_A(tQ_k)&[Q_k,\ldots,Q_k]-D^{n-1}
\varphi_A(0)[Q_k,\ldots,Q_k]-D^n\varphi_A(0)[tQ_k,Q_k,\ldots,Q_k]\\
&= (f^{(n-1)}(\lambda_k+t)-f^{(n-1)}(\lambda_k)-tf^{(n)}(\lambda_k))Q_k.
\end{align*}
Since $\varphi_A$ is $n$-times Fr\'echet differentiable, we deduce that
$$
\|(f^{(n-1)} (\lambda_k+t)-f^{(n-1)}(\lambda_k)-tf^{(n)}(\lambda_k))Q_k\|_p
= o(|t|\|Q_k\|_p)\|Q_k\|_p^{n-1},\quad |t|\to 0,
$$
uniformly in $k$.

Using $\|Q_k\|_p=1$, it follows that for every $\varepsilon>0$ 
there exists $\eta>0$ such that if $0 < |t| < \eta$ then
\begin{equation*}
|f^{(n-1)}(\lambda_k+t)-f^{(n-1)}(\lambda_k)-tf^{(n)}(\lambda_k)| 
\leq \varepsilon |t|,\quad k\geq 0.
\end{equation*}
Hence, for $0 < |t|< \eta$ and all $k\geq 0$ we have
\begin{equation*}
\left|\frac{f^{(n-1)}(\lambda_k+t)-f^{(n-1)}
(\lambda_k)}{t}-f^{(n)}(\lambda_k)\right| \leq \varepsilon.
\end{equation*}
Since $k$ is arbitrary, $\{\lambda_k\}_{k=0}^\infty$ is 
dense and both $f^{(n)}$ and $f^{(n-1)}$ are continuous, it follows that
\begin{equation*}
\left|\frac{f^{(n-1)}(s+t)-f^{(n-1)}(s)}{t}-f^{(n)}(s)\right| 
\leq \varepsilon,\quad s \in \Rl, \, 0<|t|<\eta.
\end{equation*}
Equivalently,
\begin{equation*}
\sup_{0<|s-r|<\eta}\left|
\frac{f^{(n-1)}(s)-f^{(n-1)}(r)}{s-r}-f^{(n)}(s)\right| \leq \varepsilon.
\end{equation*}
By the triangle inequality, whenever $0 < |s-r| <\eta$ we have
\begin{align*}
|f^{(n)}(s)-f^{(n)}(r)| &\leq \left|f^{(n)}(s)-\frac{f^{(n-1)}(s)-
f^{(n-1)}(r)}{s-r}\right| + 
\left|\frac{f^{(n-1)}(s)-f^{(n-1)}(r)}{s-r}-f^{(n)}(r)\right|\\
&\leq 2\varepsilon.
\end{align*}
This shows that for every $\varepsilon>0$ there exists $\eta>0$ such that
\begin{equation*}
\sup_{0<|s-r|<\eta} |f^{(n)}(s)-f^{(n)}(r)| \leq 2\varepsilon.
\end{equation*}
That is, $f^{(n)}$ is uniformly continuous.
\end{proof}

The following simple lemma is well-known, but we supply a proof for convenience.

\begin{lemma}\label{convolution_lemma}
Let $\phi$ be a smooth compactly supported function supported 
in the interval $(-1,1)$ such that $\int_{-\infty}^\infty \phi(s)\,ds =1.$ Denote
\begin{equation*}
\phi_{\varepsilon}(t) := \varepsilon^{-1} \phi(\varepsilon^{-1} t),\quad \varepsilon>0.
\end{equation*}
If $f \in C_b^n(\Rl)$, then for all $\varepsilon>0$ we have 
$\phi_{\varepsilon}\ast f \in C^\infty_b(\Rl)$. 
If in addition $f^{(n)}$ is uniformly continuous, then 
\begin{equation*}
 \lim_{\varepsilon\to 0} \|f^{(n)}-(\phi_{\varepsilon}\ast f)^{(n)}\|_\infty = 0.
\end{equation*}
\end{lemma}

\begin{proof}
The assertion that $\phi_{\varepsilon}\ast f \in C^\infty_b(\Rl)$ is clear
from the assumption on $\phi$.
Note further that for all $f \in C_b^n(\Rl)$,
\begin{equation*}
(\phi_{\varepsilon}\ast f)^{(n)} = \phi_{\varepsilon}\ast f^{(n)}.
\end{equation*}
Hence, it suffices to take $n=0$ and prove only that 
if $f$ is uniformly continuous then
\begin{equation}\label{n=0_case}
\lim_{\varepsilon\to 0} \|f-\phi_{\varepsilon}\ast f\|_{\infty} = 0.
\end{equation}
Since $\int_{-\infty}^\infty \phi(s)\,ds = 1$, 
it follows that $\int_{-\infty}^\infty \phi_{\varepsilon}(s)\,ds = 1$ and hence
\begin{equation*}
f(t)-(\phi_{\varepsilon}\ast f)(t) = 
\int_{-\infty}^\infty \phi_{\varepsilon}(t-s)(f(t)-f(s))\,ds,\quad t \in \Rl.
\end{equation*}
Since $\phi$ is supported in the set $(-1,1)$, $\phi_{\varepsilon}$ 
is supported in $(-\varepsilon,\varepsilon)$ and therefore
\begin{equation*}
f(t)-(\phi_{\varepsilon}\ast f)(t) = 
\varepsilon^{-1}\int_{t-\varepsilon}^{t+\varepsilon} 
\phi(\varepsilon^{-1} (t-s))(f(t)-f(s))\,ds,\quad t \in \Rl.
\end{equation*}
By the triangle inequality,
\begin{equation*}
|f(t)-(\phi_{\varepsilon}\ast f)(t)| 
\leq 2\sup_{s \in \Rl, |t-s|<\varepsilon}|f(s)-f(t)|,\quad t \in \Rl.
\end{equation*}
Therefore,
\begin{equation*}
\|f-\phi_{\varepsilon}\ast f\|_{\infty} 
\leq 2\sup_{t,s \in \Rl, |t-s|<\varepsilon} |f(s)-f(t)|.
\end{equation*}
Due to the uniform continuity of $f$, the right hand-side goes to $0$ when 
$\varepsilon\to 0$.
Therefore (\ref{n=0_case}) holds true.
\end{proof}

\begin{proof}[Proof that \eqref{continuity_condition} implies 
\eqref{continuous_frechet_differentiability}]
This result is an improvement of \cite[Theorem 3.4]{LeMerdySkripka2019}.
Our approach is based on the proofs of the latter theorem 
and of \cite[Lemma 3.12]{LeMerdySkripka2019}. We use
the multiple operator integrals $T^{A,\ldots,A}_{f^{[k]}}$ from the latter
paper. We let $S_k$ denote the symmetric group of degree $k$, for all $k\geq 1$.

Let $f \in C^{n}_b(\Rl)$ be such that $f^{(n)}$ is uniformly continuous, 
let $A$ be a self-adjoint operator on $H$ 
and let $\phi$ be as in Lemma \ref{convolution_lemma}. 
For brevity, denote $f_{\varepsilon} := \phi_{\varepsilon}\ast f$ for all
$\varepsilon >0$.
Then the statement of the lemma implies that for every 
$\varepsilon>0$ the function $f_{\varepsilon}$ belongs 
to $C^\infty_b(\Rl)$. Hence, by \cite[Theorem 3.3]{LeMerdySkripka2019}
$f_{\varepsilon}$ is $n$-times continuously Fr\'echet differentiable, 
and 
\begin{equation*}
D^n_p f_{\varepsilon}(A)[X_1,X_2,\ldots,X_n] = 
\sum_{\sigma \in S_n} T^{A,\ldots,A}_{f_{\varepsilon}^{[n]}}
(X_{\sigma(1)},X_{\sigma(2)},\ldots,X_{\sigma(n)}),\quad X_1,\ldots,X_n\in \Sc^{p}.
\end{equation*}

Since $f_{\varepsilon}$ is $n$-times Fr\'echet differentiable, 
for every $\delta>0$ there exists a $\delta_{\varepsilon}'>0$
such that if $\|X\|_p< \delta_{\varepsilon}'$, $X\in \Sc^{p}_{sa}$, then
\begin{align*}
&\|(D^{n-1}_pf_{\varepsilon}(A+X)-D^{n-1}_pf_{\varepsilon}(A))
[X_1,X_2,\ldots,X_{n-1}]-D^n_p f_{\varepsilon}(A)[X_1,\ldots,X_{n-1},X]\|_p\\
&\leq \delta\|X\|_p\|X_1\|_p\cdots\|X_{n-1}\|_p,
\end{align*}
for all $X_1,\ldots, X_{n-1}$ in $\Sc^{p}$.

Define
\begin{equation}\label{Gamma}
\Gamma(f)[X_1,\ldots,X_n] := 
\sum_{\sigma\in S_n} T^{A,\ldots,A}_{f^{[n]}}(X_{\sigma(1)},
\ldots,X_{\sigma(n)}),\quad X_1,\ldots,X_n \in \Sc^p.
\end{equation}
It follows from above 
that $\Gamma(f_{\varepsilon})=D_p^n f_\varepsilon(A)$
for all $\varepsilon>0.$
Thus we have
\begin{align*}
\|(D^n_pf_{\varepsilon}(A)-\Gamma(f))[X_1,\ldots,X_{n-1},X]\|_p &= 
\|(\Gamma(f_\varepsilon)-\Gamma(f))[X_1,\ldots,X_{n-1},X]\|_p\\
&\leq C_{p,n}\|(f_{\varepsilon}-f)^{(n)}\|_{\infty}\|X\|_p\|X_1\|_{p}\cdots \|X_{n-1}\|_p,
\end{align*}
by \cite[Theorem 2.2]{LeMerdySkripka2019}, where
$C_{p,n}>0$ is a constant only depending on $p$ and $n$.

Likewise, using
\cite[(3.24) $\&$ (3.25)]{LeMerdySkripka2019}, we have a similar estimate
\begin{align*}
\|(D^{n-1}_p(f-f_{\varepsilon})(A+X) & -D^{n-1}_p(f-f_{\varepsilon})(A))
[X_1,\ldots,X_{n-1}]\|_p \\
& \leq C_{p,n}\|(f_{\varepsilon}-f)^{(n)}\|_\infty\|X\|_p\|X_1\|_{p}\cdots \|X_{n-1}\|_p.
\end{align*}

Therefore, for every $\delta>0$ and every
$\varepsilon>0$, there exists a $\delta_\varepsilon'>0$ 
such that for every $X\in \Sc^{p}_{sa}$ with
$\|X\|_p < \delta_\varepsilon'$, and for all $X_1,\ldots,X_{n-1}$ in $\Sc^{p}$, we have
\begin{align*}
&\|(D^{n-1}_pf(A+X)-D^{n-1}_pf(A))[X_1,\ldots,X_{n-1}]-\Gamma(f)[X_1,\ldots,X_{n-1},X]\|_p\\
&\leq \|(D^{n-1}_pf_{\varepsilon}(A+X)-D^{n-1}_pf_{\varepsilon}(A))
[X_1,\ldots,X_{n-1}]-D^nf_{\varepsilon}(A)[X_1,\ldots,X_{n-1},X]\|_p\\
&\quad +\|(D^{n-1}_p(f-f_{\varepsilon})
(A+X)-D^{n-1}_p(f-f_{\varepsilon})(A))[X_1,\ldots,X_{n-1}]\|_p\\
&\quad +\|(D^{n}f_{\varepsilon}(A)-\Gamma(f))[X_1,\ldots,X_{n-1},X]\|_p\\
&\leq \delta\|X\|_p\|X_1\|_p\cdots\|X_{n-1}\|_p
+2C_{p,n}\|(f-f_{\varepsilon})^{(n)}\|_\infty\|X\|_p\|X_1\|_p\cdots\|X_{n-1}\|_p.
\end{align*}

Let $\delta > 0$. 
From Lemma \ref{convolution_lemma}, we may select $\varepsilon>0$ sufficiently small 
so that $\|(f-f_{\varepsilon})^{(n)}\|_\infty<\delta C_{p,n}^{-1}$.
Then we find $\delta':=\delta_\varepsilon'$ such that
\begin{align*}
\|(D^{n-1}_pf(A+X)-D^{n-1}_pf(A))[X_1,\ldots,X_{n-1}]&-\Gamma(f)[X_1,\ldots,X_{n-1},X]\|_p\\
 &\leq 3\delta \|X\|_p\|X_1\|_p\cdots\|X_{n-1}\|_p
\end{align*}
for $X\in \Sc^{p}_{sa}$ with
$\|X\|_p < \delta'$, and for all $X_1,\ldots,X_{n-1}$ in $\Sc^{p}$.
Since $\delta>0$ is arbitrary, we arrive at
\begin{align*}
\|(D^{n-1}_pf(A+X)-D^{n-1}_pf(A))[X_1,\ldots,X_{n-1}]&-\Gamma(f)[X_1,\ldots,X_{n-1},X]\|_p\\
=o(\|X\|_p)\|X_1\|_p\cdots \|X_{n-1}\|_p
\end{align*}
as $\|X\|_p\to 0$, uniformly for $X_1,\ldots,X_{n-1}\in \Sc^{p}$.
Hence $f$ is $n$-times Fr\'echet differentiable at $A$ in $\Sc^{p}$,
with 
\begin{equation}\label{Gamma2}
\Gamma(f) = D^n_pf(A).
\end{equation}

Let us now check that $X\mapsto D^n_pf(A+X)$ is continuous on $\Sc^{p}_{sa}$.
It suffices to prove continuity at $0$.
It follows from (\ref{Gamma2}), (\ref{Gamma}) 
and \cite[Theorem 2.2]{LeMerdySkripka2019}
that 
$$
\Vert D^n_p(f-f_\varepsilon)(A+X) - 
D^n_p(f-f_\varepsilon)(A)\Vert\leq 
C_{p,n}\|(f_{\varepsilon}-f)^{(n)}\|_{\infty},
$$
where the norm in the left hand-side is computed in 
the Banach space $\Bc_{n}(\Sc^p\times \cdots\times \Sc^p,\Sc^p)$.
We have
\begin{align*}
\Vert D^n_p f(A+X) - 
D^n_p f(A)\Vert & \leq \Vert D^n_p(f-f_\varepsilon)(A+X) - 
D^n_p(f-f_\varepsilon)(A)\Vert \\ & + \Vert D^n_p f_\varepsilon(A+X) - 
D^n_p f_\varepsilon(A)\Vert,
\end{align*}
and hence
$$
\Vert D^n_p f(A+X) - 
D^n_p f(A)\Vert\leq C_{p,n}\|(f_{\varepsilon}-f)^{(n)}\|_{\infty} + \Vert D^n_p f_\varepsilon(A+X) - 
D^n_p f_\varepsilon(A)\Vert.
$$
By \cite[Theorem 3.3]{LeMerdySkripka2019}, the mapping $X\mapsto D^n_pf_\varepsilon(A+X)$ is continuous
for any $\varepsilon>0$. Recall that $\|(f_{\varepsilon}-f)^{(n)}\|_{\infty} \to 0$
when $\varepsilon\to 0$. We deduce 
that
$$
\lim_{X\to 0} \Vert D^n_p f(A+X) - 
D^n_p f(A)\Vert =0,
$$
which completes the proof.
\end{proof}

\section{Final comments}
We provide two additional comments.

\smallskip
\noindent
{\it Comment 1.}
It follows from the proof of Theorem \ref{main_result} that if $f$ satisfies the conditions 
of this theorem,  then 
\begin{equation}\label{perturbation_formula}
D^n_pf(A)[X_1,\ldots,X_n] = \sum_{\sigma\in S_n} T^{A,A,\cdots,A}_{f^{[n]}}
(X_{\sigma(1)},\ldots,X_{\sigma(n)}),\quad X_1,\ldots,X_n\in \Sc^p,
\end{equation} 
for all self-adjoint operators $A$ on $H$. Further the argument 
in \cite[(3.53)]{LeMerdySkripka2019} shows the following Taylor formula:
\begin{align}\label{taylor_expansion}
f(A+X) &= f(A)+D^1_pf(A)[X]+\frac{1}{2!}D^2_pf(A)[X,X]+\cdots\\
&\quad \cdots +\frac{1}{(n-1)!}D^{n-1}_pf(A)[X,\ldots,X]+
T^{A+X,A,\ldots,A}_{f^{[n]}}(X,\ldots,X).
\end{align}

\smallskip
\noindent
{\it Comment 2.} The Potapov-Sukochev result \cite[Theorem 1]{PS-acta} 
holds true on all non-commutative $L^p$-spaces (not only on
Schatten classes). However the results discussed in this note
cannot be extended beyond Schatten classes. The lack of Fr\'echet
differentiability in a general context was already mentioned
in \cite[Introduction]{DPS}. For the sake of completeness,
we provide a simple example in
the commutative case. Fix
$1<p<\infty$ and let $f\in C^1_b(\Rl)$ such that $f(t)=t^2$
for $\vert t\vert\leq 2$. Then for any 
$X\in L^\infty[0,1]\subset L^p[0,1]$ with $\Vert X\Vert_\infty \leq 1$, we have 
$f(1+X)-f(1) =2X+X^2$. If $f$ was $1$-time Fr\'echet differentiable
in $L^p[0,1]$, $D_p^1f(1)$ would be the mapping $X\mapsto 2X$ and we would have
\begin{equation}\label{o}
\Vert X^2\Vert_p=o(\Vert X\Vert_p).
\end{equation} 
For all $\varepsilon\in (0,1)$, the indicator
function $X=\chi_{[0,\varepsilon]}$ satisfies 
$\Vert \chi_{[0,\varepsilon]}^2\Vert_p=
\Vert \chi_{[0,\varepsilon]}\Vert_p = 
\varepsilon^{\frac{1}{p}}$. This 
contradicts (\ref{o}).

\smallskip
\noindent
{\it Comment 3.} Repeating the arguments in the proof that \eqref{frechet_differentiability} implies \eqref{continuity_condition} of Theorem
\ref{main_result}, we can also see that uniform continuity of $f^{(n)}$ is a necessary condition for $f$ to be $n$-times Fr\'echet 
differentiable in $\Sc^1$ at every self-adjoint operator.

\bigskip
\noindent
{\bf Acknowledgements.} 
The first author was supported by the French 
``Investissements d'Avenir" program, 
project ISITE-BFC (contract ANR-15-IDEX-03).

\vskip 0.5cm

\end{document}